\RequirePackage{fix-cm}
\documentclass[smallextended]{svjour3}       
\smartqed  
\usepackage{graphicx}
\usepackage{amsfonts}
\usepackage{amsmath}		
\usepackage{amssymb}
\usepackage{float}
\usepackage{hyperref}
\usepackage[sort]{cite}
\usepackage{graphicx}
\usepackage{booktabs, multicol, multirow}
\usepackage{bigstrut}
\usepackage{array}
\usepackage{tabularx}
\usepackage{pdflscape}
\usepackage{adjustbox}
\usepackage{longtable}
\usepackage{makeidx}
\usepackage{rotating}
\usepackage{wasysym}
\usepackage{graphicx}
\usepackage{graphics}
\usepackage{float}
\usepackage{vector}
\usepackage {setspace}
\usepackage{latexsym}
\usepackage{amsfonts}
\usepackage{mathptmx}
\usepackage{eufrak}
\usepackage{makeidx}
\usepackage[latin2]{inputenc}
 \journalname{}
\begin{document}

\title{Structured two-point  stepsize gradient methods for nonlinear least squares
}


\author{Hassan Mohammad         \and
        Mohammed Yusuf Waziri 
}


\institute{Hassan Mohammad \at
              Department of Mathematical Sciences, Faculty of Physical Sciences, Bayero University, 700241 Kano,  Kano, Nigeria. \\
              \email{hmuhd.mth@buk.edu.ng}             \\
             \emph{Department of Applied Mathematics, IMECC, University of Campinas, 13083-970 Campinas, SP, Brazil.} 
           \and
          Mohammed Yusuf Waziri  \at
              Department of Mathematical Sciences, Faculty of Physical Sciences, Bayero University, 700241 Kano,  Kano, Nigeria.
}

\date{Received: date / Accepted: date}

\maketitle
\begin{abstract}
In this paper, we present two choices of structured spectral gradient methods for solving nonlinear least-squares problems.  In the proposed methods, the scalar multiple of identity approximation of the Hessian inverse is obtained by imposing the structured quasi-Newton condition. Moreover, we propose a simple strategy for choosing the structured scalar in the case of negative curvature direction. Using the nonmonotone line search with the quadratic interpolation backtracking technique, we prove that these proposed methods are globally convergent under suitable conditions. Numerical experiment shows that the method is competitive with some recent developed methods.
\end{abstract}

\section{Introduction}
\label{intro}
Nonlinear least-squares are special kinds of unconstrained optimization problems that involves the minimization of the norm of the squares of twice differentiable functions called \textit{the residual functions}. These kinds of problems are solved either by general unconstrained minimization methods or by specialized methods that take into account the special structure of the objective function. Nonlinear least-squares problems arise in many applications, such as data fitting, optimal control, parameter estimation, experimental design, and imaging problems.\cite{golub2003, kim07,li2012}

Most of the iterative methods for nonlinear least-squares problems requires computation and storage of the Hessian of the objective function. The Hessian, in this case, is the sum of two terms. The first term involves only the gradients of the functions, while the second involves the second-order derivatives of the functions. In practice, computing the complete Hessian is too expensive  in most practical problems (see, \cite{seven, sunyuan}) because it requires the exact second-order derivatives of the residual function, which are rarely available at a reasonable cost.  Thus, Newton-type methods that require exact Hessian are not suitable for such problems. As a consequence,  alternative approaches that make use of the first derivative information of the function have been developed (see, \cite{nazareth80} for details), such as the Gauss-Newton method (GN) and  Levenberg-Marquardt method (LM) \cite{hartley61,levenberg44, marquardt63}. Notwithstanding, both GN and LM methods neglect the second-order part of the Hessian of the objective function  (\ref{eq2}). As a result, they are expected to perform well when solving zero residual problems (in this case the residual at the minimizer is zero). Nevertheless, when solving large residual problems, these methods may perform poorly \cite{seven}. The requirement of computing the Jacobian for the methods mentioned above also contributes to their inefficiency, in the case of large-scale problems. Thus, Jacobian-free approaches are needed for solving such cases (see, for example, \cite{knoll04, xu2012, xu2016}). Structured quasi-Newton methods (SQN) are initialized in \cite{dennis71, dennis73} to overcome the difficulties that arise when solving large residual problems using GN or LM methods. These methods combine Gauss-Newton and Quasi-Newton methods in order to make good use of the structure of the Hessian of the objective function. 

Many studies have been made to the development of the methods for minimizing the sum of squares of nonlinear functions (see, for example, \cite{albaali03, albaali85, bartholomew77, betts76, fletcher87, lukvsan96, keown75, nazareth80} and reference therein). Structured quasi-Newton methods show compelling and improved numerical performance compared to the classical methods \cite{spedicato88, wangliqi10}. However, their performance on large-scale problems is not encouraging because they require a considerable large amount of matrix storage. Based on this drawback, matrix-free approaches are preferable for solving large-scale nonlinear least-squares problems.

The two-point stepsize gradient method (also called the spectral gradient or Barzilai-Borwein method) is a matrix-free method for general unconstrained optimization problems. This method was introduced by Barzilai and Borwein \cite{bb1} as a type of gradient method where the choice of the stepsize along the negative gradient is obtained from a two-point approximation to the secant equation associated with the quasi-Newton approximation of the Hessian of the objective function. Due to its simplicity in the Hessian approximation and its efficiency, the spectral gradient method has received much attention in the past decades.

Barzilai and Borwein analyzed their method in two-dimensional quadratic functions. Later Raydan \cite{bb2} proved the global convergence of the spectral gradient method for $n$ dimensional strictly convex quadratic functions. Raydan \cite{bb3}, incorporated the Barzilai Borwein gradient method in a globalization scheme using a nonmonotone line search strategy of Grippo et. al. \cite{bb4} (GLL) for general unconstrained optimization problems. Dai and Zhang \cite{daizhang01} proposed an adaptive two-point stepsize gradient method that adaptively chooses the parameter on which the nonmonotone globalization of the spectral gradient algorithm depends. Zhang et. al. \cite{zhang1999}, proposed a modification of the famous secant equation as $B_{k+1}s_k=\bar{y}_k$, where $\bar{y}_k$ contains some information of gradients at the last two iterates and function values, different from the usual secant equation in which $y_k$ is the difference between the last two iterates only.  In \cite{bb6},  Dai et. al. presented two choices of the modified two-point stepsize gradient methods using interpolation strategy. These choices of stepsize given by Dai et.al. exploits the function and gradient information at the current and previous step, and this feature improved the performance of the classical choices of the spectral gradient parameter.  

In view of the above developments,  in this paper, we focus our study on the special structure of the objective function of the nonlinear least squares and motivated by the success of the spectral gradient method introduced by Barzilai and Borwein in \cite{bb1}, the global Barzilai and Borwein gradient method presented by Raydan in \cite{bb3} for general unconstrained optimization problems, and the nonmonotone globalization strategy of Zhang and Hager \cite{zhang04}. We proposed alternative approximations of the complete Hessian using a structured spectral parameter and obtained Hessian-free algorithm, achieved by exploiting the special structure of the Hessian of the objective function.

The remaining part of this paper is organized as follows. In Section 2 we discuss some preliminaries. Section 3 describes the general algorithm of the proposed methods, and the global convergence results. Numerical experiments are provided in Section 4. Concluding remarks and future works are given in Section 5. Unless otherwise stated, throughout this paper $\|.\|$ stands for the Euclidean norm of vectors, and $\|.\|_{F}$ stands for the Frobenius norm of matrices. $g_k=\nabla g(x_k), F_k=F(x_k),$ and $f_k=f(x_k)$.

\section{Preliminaries}

Consider the problem of determining a minimizer $x^* \in \mathbb{R}^n$ of the nonlinear-least squares problem (LS) of the form
\begin{equation}\label{eq2}
\displaystyle \min _{x\in \mathbb{R}^n}f(x), ~ f(x)=\frac{1}{2}\|F(x)\|^{2}=\frac{1}{2}\sum _{i=1}^{m}(F_{i}(x))^2.
\end{equation}
Here $F:\mathbb{R}^n \to \mathbb{R}^m \text{ } (m\ge n)$, $F(x)=(F_{1}(x), F_{2}(x),...,F_{m}(x))^T$, $F_i: \mathbb{R}^n \to \mathbb{R}$. $F$ assumed to be twice continuously differentiable function. 

The gradient and Hessian of the objective function (\ref{eq2}) have a special form and are given by
\begin{equation}\label{eq3}
\nabla f(x)=\sum _{i=1}^m F_{i}(x) \nabla F_i(x)=J^TF,
\end{equation}
\begin{align}\label{eq4}
\nabla ^2f(x)& =\sum  _{i=1}^m \nabla F_{i}(x)\nabla F_{i}(x)^T+\sum_{i=1}^m F_{i}(x)\nabla ^2F_{i}(x)\\
&= J^TJ+S,
\end{align}
where $F_{i}$ is the $i$-th component of $F$ and $\nabla ^2F_{i}(x)$ is its Hessian, $J=\nabla F^T$ is the Jacobian of the residual function $F$ at $x$ and $S$ is a square matrix representing the second term of (\ref{eq4}).

In \cite{bb1} Barzilai and Borwein (BB), a two-point step size gradient method for unconstrained optimization problem was developed. Considering the solution of the problem
\begin{equation}\label{eq9}
\min _{x\in \mathbb{R}^n}f(x),
\end{equation}
where $f:\mathbb{R}^n\to \mathbb{R}$ is a continuously differentiable function. The BB method is iteratively given by
\begin{equation}\label{eq10}
x_{k+1}=x_k+s_k, 
\end{equation}
where $s_k=-\alpha _{k}g_k,$ and the scalar $\alpha _k$ (called the spectral parameter) is regarded as the approximation of the Hessian inverse of $f$ at $x_k$. Let $D_k=\alpha _kI$, imposing quasi-Newton condition on $D_k$, such that $D_k^{-1}s_{k-1}\approx y_{k-1}$, Barzilai and Borwein considered minimizing $\|D_k^{-1}s_{k-1}-y_{k-1}\|$ and obtained
\begin{equation}\label{eq11}
 \qquad \alpha ^{1} _{k}=\frac{s_{k-1}^Ts_{k-1}}{s_{k-1}^Ty_{k-1}}, \qquad k\ge 1.
\end{equation}
They also consider another choice of minimizing $\|s_{k-1}-D_ky_{k-1}\|$ and obtained
\begin{equation}\label{eq11b}
 \qquad \alpha ^{2} _{k}=\frac{s_{k-1}^Ty_{k-1}}{y_{k-1}^Ty_{k-1}}, \qquad k\ge 1,
\end{equation}
where $s_{k-1}=x_{k}-x_{k-1},\text{ } y_{k-1}=g_{k}-g_{k-1}.$ 

We have noticed some further studies on the spectral gradient methods for unconstrained optimization in \cite{daizhang01, bb10,  biglari13, liu17, liu18, curtis16, kafaki13} and reference therein. However, we believe that the classical choices of the stepsize need to be explored in order to understand the basic idea associated with the study of the structured spectral gradient methods for nonlinear least squares problems. Therefore, in this paper, we based our study on the two choices (\ref{eq11}) and (\ref{eq11b}) used in \cite{bb1} and proposed two structured choices of spectral parameters for solving nonlinear least-squares problems.


\section{Structured spectral algorithm and convergence result}

Let $H_k, ~~ k\ge 1$ be the approximation of the structured Hessian (\ref{eq4}) such that $H_k$ satisfy the quasi-Newton equation
\begin{equation}\label{sqnc1}
    H_ks_{k-1}=z_{k-1},
\end{equation}
where $z_{k-1}$ is a structured vector different from $y_{k-1}$ that can be determine as follows:

By Taylor's series expansion of $F_i(x_{k-1})\text{ and }\nabla F_i(x_{k-1})$ at current step $x_k$ neglecting the higher order terms, we have
\begin{equation}\label{eqts1} F_i(x_{k})-F_i(x_{k-1})\approx \nabla F_i(x_{k})^Ts_{k-1},\end{equation}
and
\begin{equation}\label{eqts2} \nabla F_i(x_{k})- \nabla F_i(x_{k-1})\approx  \nabla ^2F_i(x_{k})s_{k-1}.\end{equation}
Pre-multiply  (\ref{eqts1}) and (\ref{eqts2}) by  $\nabla F_i(x_{k})\text{ and }F_i(x_{k})$ respectively, summing from $i=1\text{ to }m$ we have

\begin{equation}\label{eqts11}  \sum _{i=1}^m\nabla F_i(x_{k})\bigl(F_i(x_{k})-F_i(x_{k-1})\bigr)\approx  \sum _{i=1}^m\nabla F_i(x_{k})\nabla F_i(x_{k})^Ts_{k-1},\end{equation}
and
\begin{equation}\label{eqts21}  \sum _{i=1}^m F_i(x_{k})\bigl(\nabla F_i(x_{k})- \nabla F_i(x_{k-1})\bigr)\approx \sum _{i=1}^m F_i(x_{k}) \nabla ^2F_i(x_{k})s_{k-1}.\end{equation}
Adding  (\ref{eqts11}) and (\ref{eqts21})
\begin{align*}
    & \sum _{i=1}^m\nabla F_i(x_{k})\nabla F_i(x_{k})^Ts_{k-1}+\sum _{i=1}^m F_i(x_{k}) \nabla ^2F_i(x_{k})s_{k-1}\\
    & \approx \sum _{i=1}^m\nabla F_i(x_{k})\bigl(F_i(x_{k})-F_i(x_{k-1})\bigr)+ \sum _{i=1}^m F_i(x_{k})\bigl(\nabla F_i(x_{k})- \nabla F_i(x_{k-1})\bigr)
\end{align*}
We obtain the structured quasi-Newton condition given by
\begin{equation*}
    H_ks_{k-1}=z_k, 
\end{equation*}
where
\begin{equation}\label{sqnc11}
    z_k=J_k^T(F_k-F_{k-1})+(J_k-J_{k-1})^T F_k=2g_k-({r_{k}}+{r_{k-1}}),
\end{equation}
where $r_k=J_k^TF_{k-1},~~ r_{k-1}=J_{k-1}^TF_k$, are called \emph{structured vectors}.

Notice that unlike $y_k$, the vector $z_k$ is the difference between twice the current gradient vector and the sum of two structured gradient vectors. These two structured gradient vectors are special because $r_k$ is the product of the Jacobian transpose at $x_k$ and the residual function at $x_{k-1}$, while $r_{k-1}$ is the product of the Jacobian transpose at $x_{k-1}$ and the residual function at $x_k$. Now we are ready to give our choices of structured spectral parameters. 

Consider the following unconstrained minimization problem 
\[\min \|D_k^{-1}s_{k-1}-z_{k-1}\|=\min _{\alpha \in \mathbb{R}}\|\frac{1}{\alpha _k}s_{k-1}-z_{k-1}\|, ~~ k\ge 1, \]
and 
\[\min \|s_{k-1}-D_kz_{k-1}\|=\min _{\alpha \in \mathbb{R}}\|s_{k-1}-\alpha z_{k-1}\|, ~~ k\ge 1, \]
where $z_{k-1}$ is given by (\ref{sqnc1}). The solution of the two unconstrained minimization problems are 
\begin{equation}\label{ssp1}
   \alpha ^{*} _{k}=\frac{s_{k-1}^Ts_{k-1}}{s_{k-1}^Tz_{k-1}}.
\end{equation}

\begin{equation}\label{ssp2}
   \tilde {\alpha} _{k}=\frac{s_{k-1}^Tz_{k-1}}{z_{k-1}^Tz_{k-1}}.
\end{equation}
Notice that the only difference between  choices (\ref{ssp1})-(\ref{ssp2}) and (\ref{eq11})-(\ref{eq11b}) is that the former neglects the structure, while the latter exploit the structure of the Hessian approximation.

In order to avoid the negative curvature direction, the term $s_{k-1}^Tz_{k-1}$ appearing in the two choices of the structured stepsizes must be strictly positive, i.e. $s_{k-1}^Tz_{k-1}>0$, which is not the case all the times. In most of the variants of spectral gradient method, the negative curvature occurrence is prevented by simply setting the stepsize to a large positive number say $\lambda _{max}$ (\textit{classical approach}). This choice of safeguard may cause additional demand for backtrackings in search of suitable steplength for some test functions (see the included numerical experiments). In this regard, Curtis and Gou in \cite{curtis16} investigated another safeguard for nonpositive curvature with the emphasis on nonconvex optimization problems.
For general unconstrained optimization, Luengo and Raydan \cite{luengo03} argued that the previous stepsize may provide some information for the current stepsize and presented a \textit{retarding technique} in such a way that if $s_{k-1}^Ty_{k-1}\le 0$, then the stepsize is set to $\alpha _k=\delta \alpha _{k-1},$ where $\delta$ is a positive parameter. In a recent paper \cite{liu17}, the authors adopt this technique and obtained a reasonable improvement in the performance profile based on the number of functions evaluations and CPU time, and a slight improvement in terms of the number of iterations. Even though the retarding technique was shown to be effective based on the numerical experiment in \cite{liu17}, we feel that there may be better choices for the stepsize in the case of negative curvature direction. In what follows, we present another different choice as a remedy for negative curvature direction.

The proposed structured spectral gradient parameters in (\ref{ssp1}) or (\ref{ssp2})  is negative if and only if
\begin{equation} \label{negcuv}
s_{k-1}^Tz_{k-1}\leq0, 
\end{equation}
in this case, it is better to choose a positive term that carries some information about the current stepsize. More specifically, a simple positive modification of the term $s_{k-1}^Tz_{k-1}$ may be the best choice. We now present a simple strategy for choosing a replacement for the term $s_{k-1}^Tz_{k-1}$ in case it is negative.

If $s_{k-1}^Tz_{k-1}\leq0$, then $s_{k-1}^Tz_{k-1}+\mu \ge 0,$ where $\mu$ is a positive number greater than or equal to $|s_{k-1}^Tz_{k-1}|$. 

Now, using the Cauchy-Schwartz inequality,
\[ |s_{k-1}^Tz_{k-1}|\le \|s_{k-1}\|\|z_{k-1}\| ~~~\forall s_{k-1},~z_{k-1} \in \mathbb{R}^n .\]
Therefore, $s_{k-1}^Tz_{k-1}< 0 \implies s_{k-1}^Tz_{k-1}+ \|s_{k-1}\|\|z_{k-1}\|\ge 0$.

In general, if $s_{k-1}z_{k-1}<0$, then we set its replacement term as
\begin{equation} \label{replc1}
    \tau _k= \max \bigl\{ \beta \alpha _{k-1} ,~s_{k-1}^Tz_{k-1}+ \|s_{k-1}\|\|z_{k-1}\|\bigr\},
\end{equation}
where $\alpha _{k-1}$ is the previous stepsize, $\beta$ is a positive real number. For the first choice of the stepsize, $\tau _k$ will replace the denominator while for the second choice of the stepsize $\tau _k$ will replace the numerator.  

Since this new choice of stepsize incorporate more information about the curvature $s_{k-1}^Tz_{k-1}$, it is reasonable to expect that it is more efficient than setting the stepsize to a large positive number or using the retarding technique only. Our preliminary numerical experiments support these expectations.   

To globalize our algorithm, we use the efficient nonmonotone line search by Zhang and Hager \cite{zhang04}. In this line search, if $d_k$ is the sufficiently descent direction of $f\text{ at }x_k$, then the step length $t$ satisfies the following Armijo-type condition:

\begin{equation}\label{zhls}
f(x_{k}+td_k)\le C_k+\ \gamma g_k^Td_k, \qquad \gamma \in (0,1),
\end{equation}
 where
\begin{equation}\label{eq15}
 C_0=f(x_0), \qquad  C_{k+1}=\frac{\eta _kQ_kC_k+f(x_{k+1})}{Q_{k+1}},
\end{equation}
where $Q_0=1$, $Q_{k+1}=\eta _{k}Q_{k}+1$, $\eta_{k}\in [0,1].$

The sequence $\{C_k\}$ is a convex combination of the function values $f(x_0), f(x_1),...$,\\ $f(x_k)$. The choice of the parameter $\eta _k$ controls the degree of monotonicity. If $\eta _k=0$ for each $k$, then the line search is purely monotone (Armijo-type), if not, it is nonmonotone.

Based on these choices of the structured spectral parameters and the choice of nonmonotone line search, we now formally present the step of the structured two-point stepsize gradient algorithm for nonlinear least-squares problems as follows:\\\\
\noindent \textbf{ Algorithm: Structured two-point stepsize gradient method (SSGM) }\\
\textbf{Step 0.}  Given $x_0\in \mathbb{R}^n, \lambda _{max}\ge \lambda _{min}>0, \gamma \in (0,1),  0\le \eta _{min}\le \eta _{max}\le 1,$  $ \varepsilon >0$  and $\beta >0$.  

Set $ C_0=f_0 , Q_0=1 ~ \lambda_0=1, ~  \text{ and }k=0$. \\
 \textbf{Step 1.} $\text{~If~} \|g_k\|\le \varepsilon $ stop. \\
\textbf{Step 2.} Set $d_k=-\lambda_kg_k.$ \\
\textbf{Step 3.} Nonmonotone line search 

\textbf{Step 3.1} Set $t=1$

\textbf{Step 3.2} If
 \begin{equation}\label{eqbk}
f(x_k+td_k)\le C_k+\gamma tg_k^Td_k,
\end{equation}

 holds, then set $t_k=t $ and compute the next iterate $x_{k+1}=x_k+t_kd_k.$ Go to Step 4.  
 
 Else  compute $t$ using quadratic interpolation.  \\
\textbf{Step 4.1} Compute the stepsize by one of the choices (\ref{ssp1}) and (\ref{ssp2}). If the stepsize is less than zero,

then use (\ref{replc1}) to replace the negative term according to the choice of the stepsize.\\
\textbf{Step 4.2} Compute  
\begin{equation}\label{finallm}
\lambda _{k}=\min \bigl \{\max \bigl \{ \alpha ^{c}_{k}, \lambda _{min}\bigr\},\lambda _{max}\bigr \},
\end{equation}
where $\alpha ^{c}_k$ is the choice of the stepsize.\\
\textbf{Step 5.}
Choose $\eta _k\in [\eta _{min}, \eta _{max}]$ and compute $Q_{k+1} \text{ and }C_{k+1}$ using Equation (\ref{eq15}).\\
\textbf{Step 6.} Let $k=k+1$ and go to \textbf{Step 1}.\\

\noindent \textbf{Remarks}
\begin{itemize}
\item [1.] The reasons behind computing $\lambda _k$ using (\ref{finallm}) in Step 4 are
\begin{itemize}
\item [i.] to avoid unwanted uphill direction, and;
\item [ii.] to ensure that the sequence $\{\lambda _k\}$ is uniformly bounded for each $k$. It is clear that $\lambda _{min}\le \lambda _k\le \lambda _{max}~\forall k$.
\end{itemize}
\item [2.] The search direction $d_k=-\lambda _kg_k$ satisfies the direction assumptions \cite{zhang04}:
\begin{itemize}
\item [i.] $g_{k}^Td_k\le -c_1\|g_k\|^2$, and
\item [ii.] $\|d_k\|\le c_2\|g_k\|$,  for all k, where $c_1,c_2$ are positive constants.
\end{itemize}
\item [3.] The stepsize can be chosen as any of the stepsize choices, and a different choice entails a different method. We denote by SSGM1 the algorithm with the stepsize choice (\ref{ssp1}) and by SSGM2 the algorithm with the stepsize choice (\ref{ssp2}).
\item[4.] For computing the gradient and the structured gradients we take advantage of the structure of the problem and obtained the gradient vector as a matrix-vector product without explicit knowledge of the Jacobian matrix.
\end{itemize}

We now turn to analyze the global convergence of SSGM algorithm. First we state the following assumptions:\\
\textbf{Assumption 1}

The level set $\mathcal{L}=\{ x\in \mathbb{R}^n : f(x)\le f(x_0)\}$ is bounded.\\
\textbf{Assumption 2}
\begin{itemize}
    \item [(a)]  In some neighborhood $\mathcal{N}$ of $ \mathcal{L}$, $f$ is continuously differentiable, and  \item [(b)] the Jacobian $J(x)=\nabla F(x)^T$ is Lipschitz continuous on $\mathcal{L}$.
\end{itemize}
Also, we get from the above assumptions that there are $l_2, \mu _1, \mu _2 \mu _3 $ positive such that

\begin{equation}\label{assumpa}
\|J(x)\|<\mu _1, ~~\|F(x)\|< \mu _2, 
\end{equation}
\begin{equation}\label{assumpb}
   \|F(x)-F(y)\|\le l_2\|x-y\|, \text{ and }
\end{equation}
\begin{equation}\label{assumpc}
\|g(x)-g(y)\|\le l_3\|x-y\|, \quad \|g(x)\|\le \mu _3 ~~ \forall x,y \in \mathcal{L}. 
\end{equation}
\begin{lemma}\label{lm1}
If the Assumptions 1 and 2 holds, then there exists $M>0$, such that 
\begin{equation}\label{eqboundz}
   \|z_k\|\le M\|s_k\|. 
\end{equation}
\end{lemma}

\begin{proof}
\begin{equation*}
\begin{split}
\|z _k\|&=\|J_k^T(F_k-F_{k-1})+(J_k-J_{k-1})^TF_k\|\\
&\le\|J_k^T(F_k-F_{k-1})\|+\|(J_k-J_{k-1})^TF_k\| ~~\text{by triangular inequality },\\
& \le \|J_k\|\|(F_k-F_{k-1})\|+\|(J_k-J_{k-1})\|\|F_k\| ~~\text{induced matrix norm property },\\
& \le \mu _1l_2\|s_k\|+l _1\mu _2\|s_k\| ~~\text{from assumption $2(a)$, and Equations (\ref{assumpa}) - (\ref{assumpc}), }\\
&= (\mu _1l_2+l _1\mu _2)\|s_k\|.
\end{split}
\end{equation*}
This implies that there exists $M= \mu _1l_2+l _1\mu _2>0$ such that (\ref{eqboundz}) holds.
\end{proof}

\begin{lemma} \label{lm2}
The computed stepsize $\lambda _k$ is positive and bounded for all $k$. 
\end{lemma}

\begin{proof}
From Step $0$ of the SSGM algorithm, if $k=0,~~\lambda _1=\alpha _{0}=1$.

When $k\ge 1$ (Step 4.1 of the SSGM algorithm), we have two choices of stepsize.\\
\texttt{Case 1:} First choice $\alpha _k=\frac{\|s_{k-1}\|^2}{s_{k-1}^Tz_{k-1}}.$

If $s_{k-1}^Tz_{k-1}$ is positive, then it is obvious that $\lambda _k$ is positive. Otherwise, by (\ref{replc1}) 
\begin{equation*}
    \alpha _k=\frac{\|s_{k-1}\|^2}{\max \bigl\{ \beta \alpha _{k-1} ,~s_{k-1}^Tz_{k-1}+ \|s_{k-1}\|\|z_{k-1}\|\bigr\}}
\end{equation*}

If $\beta \alpha _{k-1} >s_{k-1}^Tz_{k-1}+ \|s_{k-1}\|\|z_{k-1}\|$, then $\alpha _k>0$ and consequently, $\lambda _k>0$. 

Otherwise, \[s_{k-1}^Tz_{k-1}+ \|s_{k-1}\|\|z_{k-1}\|\le 2\|s_{k-1}\|\|z_{k-1}\|.\]

Thus, 
\begin{equation*}
\begin{split}
\alpha _k &=\frac{\|s_{k-1}\|^2}{s_{k-1}^Tz_{k-1}+ \|s_{k-1}\|\|z_{k-1}\|}\\
& \ge \frac{\|s_{k-1}\|^2}{2\|s_{k-1}\|\|z_{k-1}\|}\\
& = \frac{\|s_{k-1}\|}{2\|z_{k-1}\|}\\
& \ge \frac{1}{2M} >0.
\end{split}
\end{equation*}
The second inequality follows from Lemma \ref{lm1}. Therefore, $\lambda _k>0$.\\
\texttt{Case 2:} Second choice $\alpha _k=\frac{s_{k-1}^Tz_{k-1}}{\|z_{k-1}\|^2}.$

If $s_{k-1}^Tz_{k-1}$ is positive, then it is obvious that $\lambda _k$ is positive. Otherwise, by (\ref{replc1}) 
\begin{equation*}
    \alpha _k = \frac{\max \bigl\{ \beta \alpha _{k-1} ,~s_{k-1}^Tz_{k-1}+ \|s_{k-1}\|\|z_{k-1}\|\bigr\}}{\|z_{k-1}\|^2}
\end{equation*}

If $\beta \alpha _{k-1} >s_{k-1}^Tz_{k-1}+ \|s_{k-1}\|\|z_{k-1}\|$, then $\alpha _k>0$ and consequently, $\lambda _k>0$. 

Otherwise,
\begin{equation*}
    \begin{split}
        \alpha _k &=\frac{s_{k-1}^Tz_{k-1}+ \|s_{k-1}\|\|z_{k-1}\|}{\|z_{k-1}\|^2}\\
        & \ge \frac{\|s_{k-1}\|\|z_{k-1}\|}{\|z_{k-1}\|^2} ~~\text{ since } s_{k-1}^Tz_{k-1}<0\\
        &=  \frac{\|s_{k-1}\|}{\|z_{k-1}\|}\\
        & \ge \frac{1}{M} >0
    \end{split}
\end{equation*}
Again the second inequality follows from Lemma \ref{lm1}.

It follows that any of the two choice of the stepsize $\lambda _k$ is positive. As already stated in the remarks, Step 4.2 of the SSGM algorithm ensures that $\lambda _k$ is bounded. 
\end{proof}
In Step 5 of the SSGM algorithm, $C_k$ is computed as the convex combination of the function values $f_0,f_1,...,f_k.$ The following Lemma established that for any choice $\eta _k\in [0,1]$, $C_k$ lies between $f_k\text{ and }A_k,$ where
\begin{equation}\label{eqA}
A_k=\frac{1}{k+1}\sum _{i=0}^kf_i.
\end{equation}
\begin{lemma}\label{lm3}
The iterates generated by the SSGM algorithm satisfy $f_k\le C_k\le A_k, \text{ }\forall k\ge 0,$ where $A_k$ is given by (\ref{eqA}).
\end{lemma}
\begin{proof}
By (\ref{eqbk}) and Remark 2(i), we have
\begin{equation}\label{eq20}
f(x_{k-1}-\lambda _k g_{k-1})=f_k\le C_{k-1}.
\end{equation}
Define a real valued function $h:\mathbb{R}\to \mathbb{R}$ by\[ h(r)=\frac{rC_{k-1}+f_k}{r+1},\] then \[\frac{dh}{dr}=\frac{C_{k-1}-f_k}{(r+1)^2}.\]
Now, since $f_k\le C_{k-1}$, this makes $h$ to be nondecreasing function, and $h(0)=f_k\le h(r), ~~ \forall r\ge 0.$ In particular, taking $r=\eta _{k-1}Q _{k-1}$,  gives \[ f_k=h(0)\le h( \eta _{k-1}Q_{k-1})=C_k.\] This implies the lower bound $f_k\le C_k, ~\forall k$.

The upper bound for $C_k$ is obtained by induction. Since $C_0=f_0,$ it holds that $C_0=A_0$ for $k=0$. Assume $C_k\le A_k$, for all $j=1,2,...,k-1$. Using the fact that $\eta _k\in [0,1]$ and $Q_0=1$ by (\ref{eq15}), we have
\begin{equation}\label{ieq30}
Q_{j}=1+\sum _{i=0}^{j}\prod _{n=0}^{i}\eta _{j-n}\le j+1.
\end{equation}
From the definition of $h$, and using inequality (\ref{ieq30}) , we have
\begin{equation}
C_k=h(\eta _{k-1}Q_{k-1})=h(Q_k-1)=h(k).
\end{equation}
By the induction step, \[ h(k)=\frac{kC_{k-1}+f_k}{k+1}\le \frac{kA_{k-1}+f_k}{k+1}=A_k.\]
Therefore, $C_k\le A_k, ~~\forall k$.
\end{proof}

\begin{theorem}
Let $f(x)$ be given by Equation (\ref{eq2}) and suppose Assumption 1 and 2 hold. Then the   sequence of iterates $\{x_k\}$ generated by the SSGM algorithm is contained in the level set $\mathcal{L}$ and
\begin{equation}\label{eq21}
\liminf _{k\to \infty}\|g_k\|=0.
\end{equation}
Moreover, if $\eta _{max}<1,$ then
\begin{equation}\label{eq22}
\lim_{k\to \infty}\|g_k\|=0.
\end{equation}
\end{theorem}
\begin{proof}
The proof of this theorem follows directly from \cite{zhang04}.
\end{proof}
%
\section{Numerical Experiments}
%
%

In this section we present some numerical experiments to assess the efficiency of the proposed structured two-point spectral gradient choices of stepsize. We examine $30$ large-scale test problems (\textit{see Table \ref{tab2} for the list}) using $10$ different dimensions $1000, 2000, ...,10000$ together with $10$ small-scale problems giving us a total of $310$ instances. The specified dimension used for the small-scale problems can be found in Table \ref{tab1}.  

All methods were coded in MATLAB R2017a and run on a Dell PC with an Intel Corei3 processor, 2.30GHz CPU speed, 4GB of RAM and Windows 7 operating system. We chose the stopping criterion as follows:
\begin{itemize}
    \item $\|g_k\|_{\infty}\le 10^{-4}$,
    \item The number of iterations exceeds $1000$, and
    \item The number of function evaluations go beyond $2000$.
\end{itemize}

In order to test the effectiveness of the proposed safeguarding technique in the case of negative curvature direction, we compare the three different versions of SSGM1 and SSGM2 choices of stepsize. The first version (\textit{named SSGM1A and SSGM2A}) is the one that involves the classical safeguard in the case $s_{k-1}z_{k-1}\le 0$, the second version (\textit{named SSGM1B and SSGM2B}) uses the retard technique for the safeguard, and the third version (\textit{named SSGM1C and SSGM2C}) use the new proposed approach (\ref{replc1}) for safeguarding negative curvature direction. After removing those problems for which at least one method fails to satisfy the stopping criterion, $297$ are left and we consider performance for this number of problems. Figure \ref{fig1}-\ref{fig6} shows the Dolan and Mor$\acute{e}$ \cite{twentyfive} performance profiles based on the number of iterations,  the number of functions evaluations, and CPU time. 

In terms of number of iterations (Figure 1), SSGM1A performs slightly better than its counterparts SSGM1B and SSGM1C, while in terms of functions evaluation and CPU time (Figure \ref{fig2} and \ref{fig3}) SSGM1C performs moderately better than SSGM1A and SSGM1B. Figure 2 clearly shows how the classical choice (SSGM1A) suffers from more function evaluations. This drawback was also observed in \cite{liu17}. Figure 3 indicates that all the $3$ versions of SSGM1 have almost equal speed.

The results of the SSGM2 are slightly different from the results of SSGM1. In this case, the SSGM2C method performed better than SSGM2A and SSGM2B in terms of the number of iterations and function evaluations (Figure 4 and 5). Figure \ref{fig6} shows that SSGM2B is faster than SSGM2A and SSGM2C. This good performance was also observed in \cite{liu17}. Based on this, we can claim that in practice, the retarding technique is better than the other techniques for safeguarding the choice of stepsize in terms of CPU time.

In the second set of experiments, we consider the implementation of our two choices of stepsize using the new technique of safeguarding in the case of negative curvature direction (\ref{replc1}), compared with the top performer of the six choices of stepsize proposed in \cite{biglari13}, SBB4 and another top performer from the two choices of stepsize presented in \cite{wangliqi10}, SGW2. Both SBB4 and SGW2 were implemented using Zhang and Hager nonmonotone line search with $\eta _k=0.7$.

For the sake of analyzing  and comparing, we implement each of the methods using the backtracking strategy stated in the paper exactly in the manner it was presented. We chose the following parameters for the implementation SSGM1 and SSGM2 methods:   
 $\gamma =10^{-4}, \lambda _{min}=10^{-30},\lambda _{max}=10^{30}$ and $\eta _{min}=0.1, \eta _{max}=0.85 \text{ and }$ $\beta =10^3$.  We have alliviated on the recommendation by Santos and Silva \cite{santos14} for choosing the parameter $\eta _k \in (0,1)$ in Step 5 of the SSGM methods as follows:
\begin{equation}\label{eta1}
 \eta _k=0.75e^{-(\frac{k}{45})^2}+0.1.
\end{equation}

\begin{figure*}
 \includegraphics[width=10cm ]{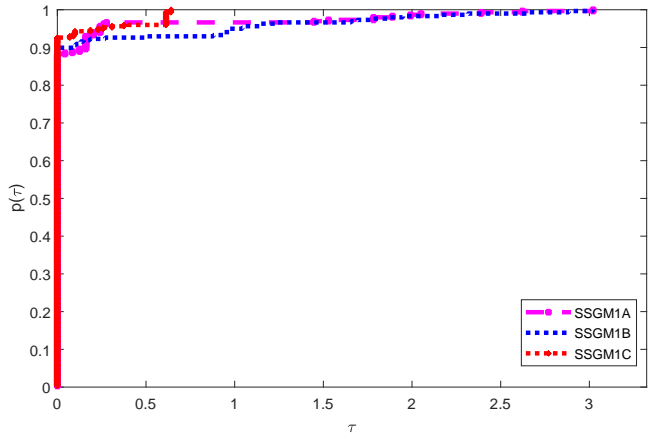}
\caption{Performance profile of SSGM1A, SSGM1B and SSGM1C
methods with respect to number of iterations}
\label{fig1}       
\end{figure*}

\begin{figure*}
\includegraphics[width=10cm]{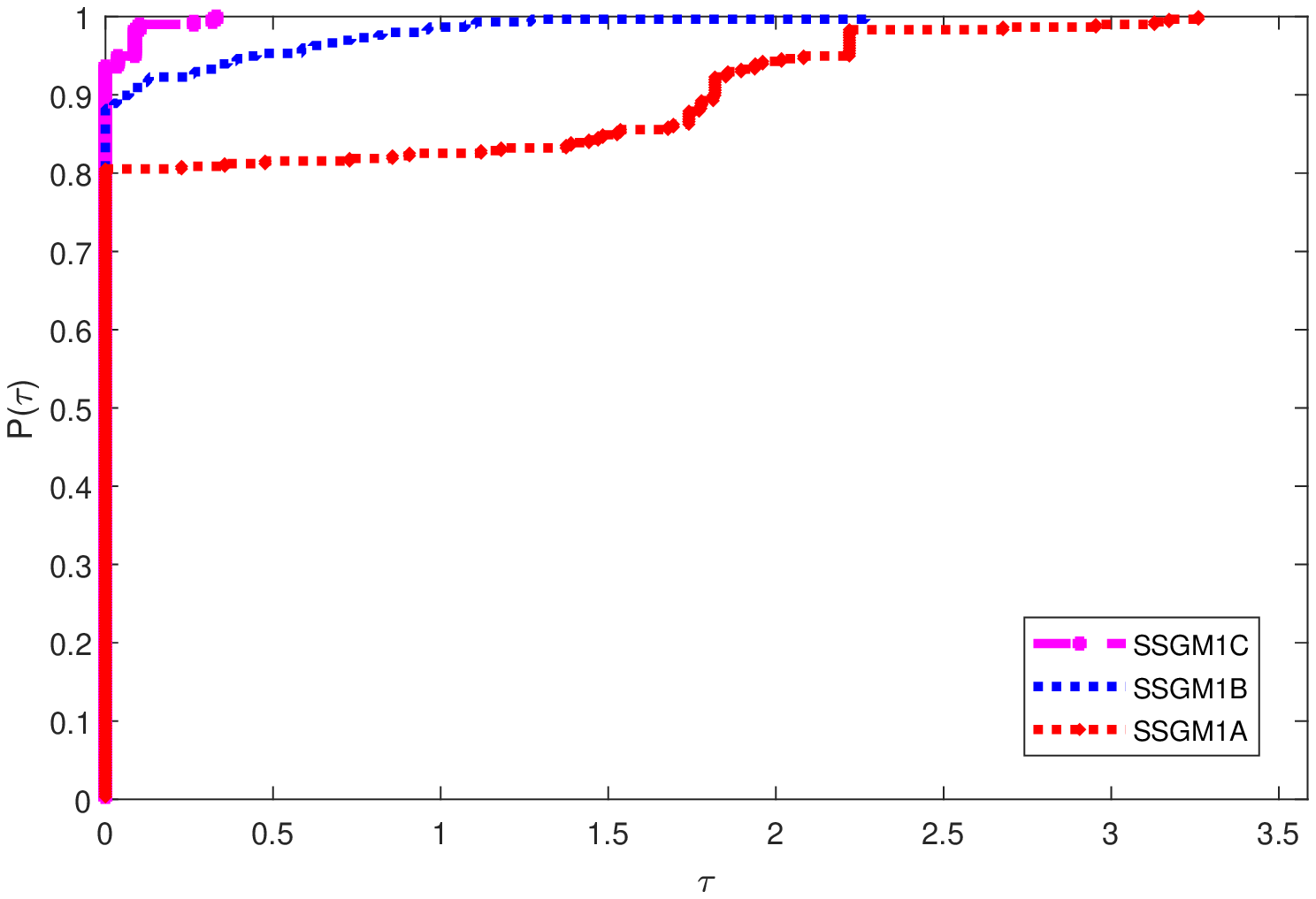}
\caption{Performance profile of SSGM1A, SSGM1B and SSGM1C
methods with respect to number of function evaluations}
\label{fig2}       
\end{figure*}

\begin{figure*}
\includegraphics[width=10cm]{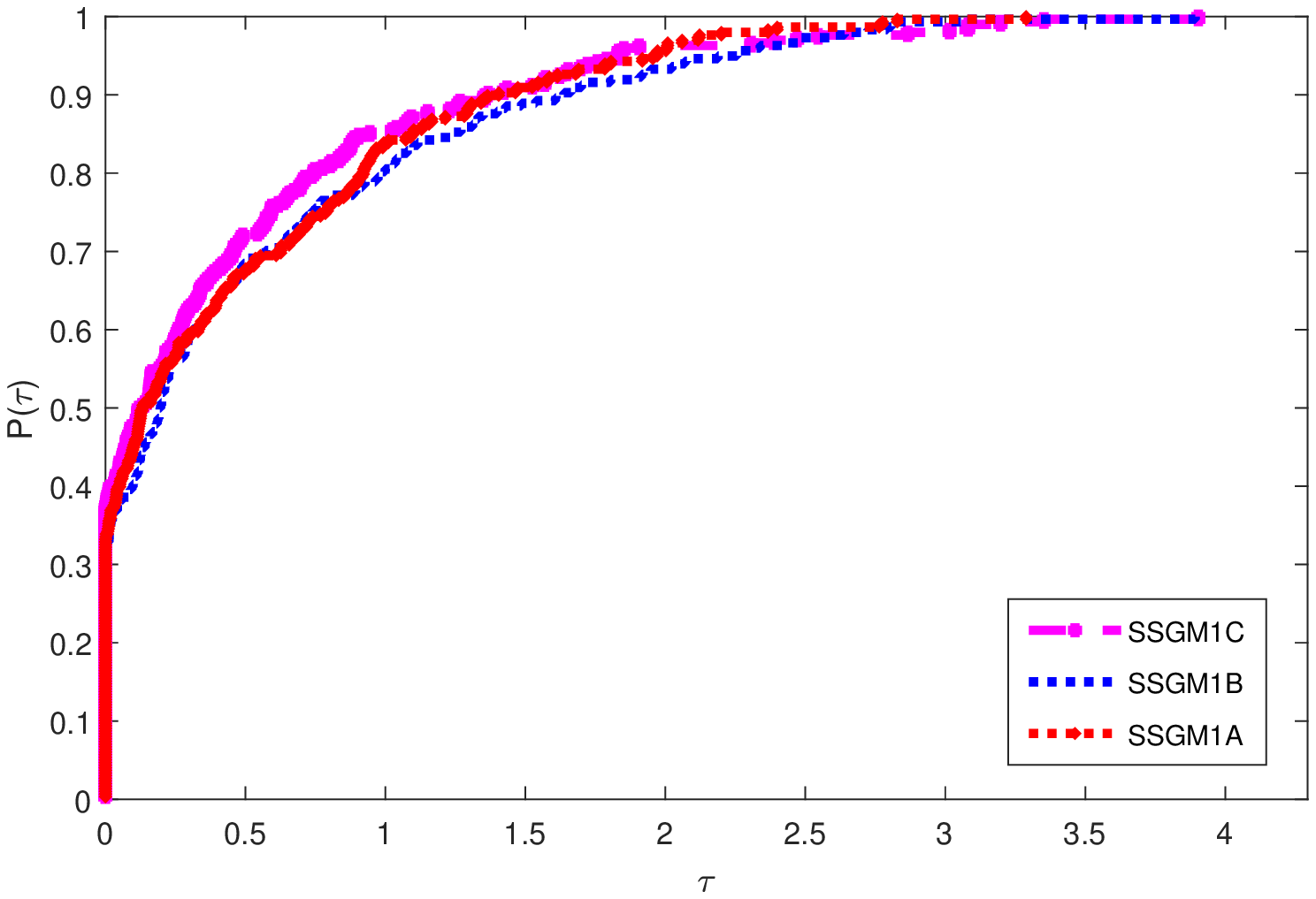}
\caption{Performance profile of SSGM1A, SSGM1B and SSGM1C
methods with respect to CPU time}
\label{fig3}       
\end{figure*}

\begin{figure*}
\includegraphics[width=10cm]{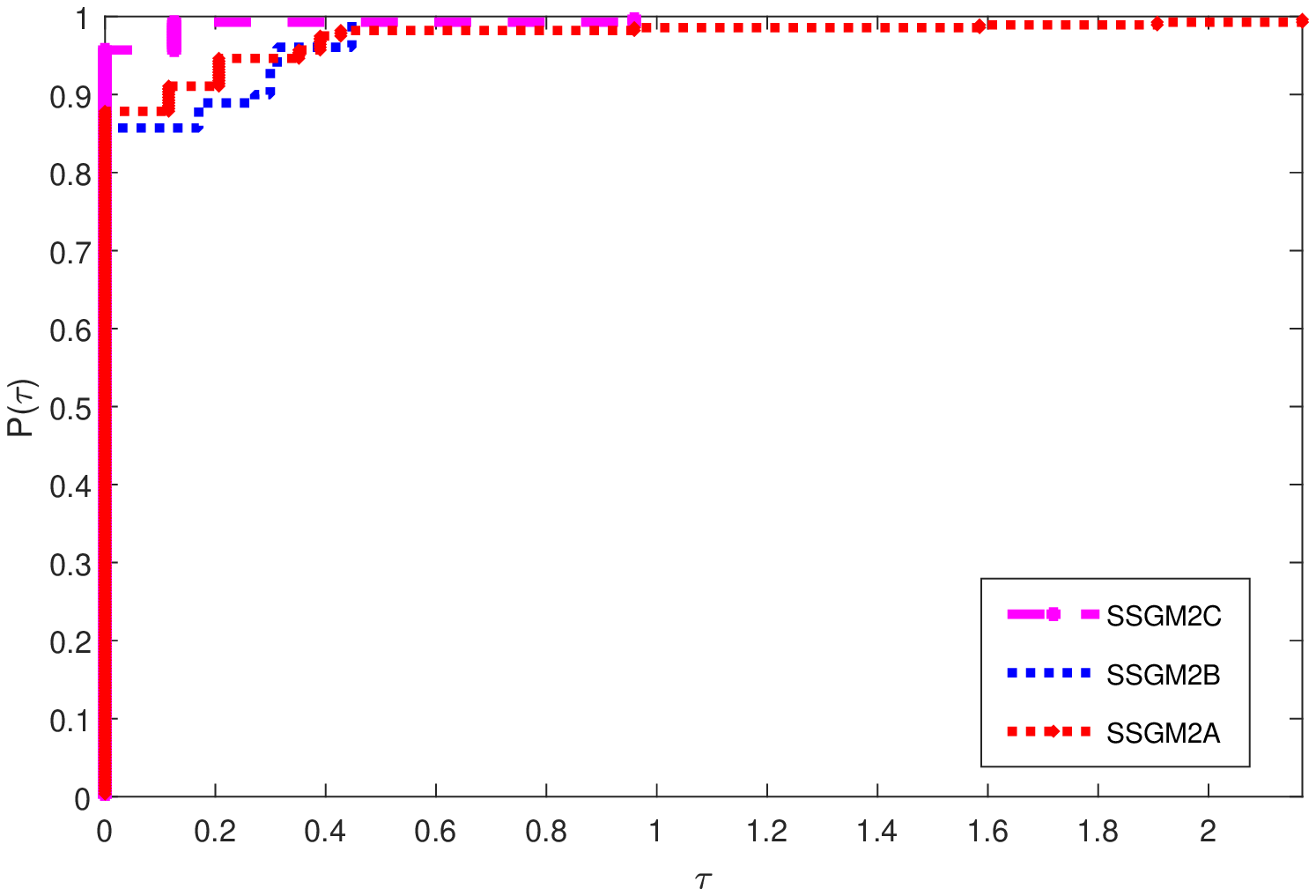}
\caption{Performance profile of SSGM2A, SSGM2B and SSGM2C
methods with respect to number of iterations}
\label{fig4}       
\end{figure*}

\begin{figure*}
\includegraphics[width=10cm]{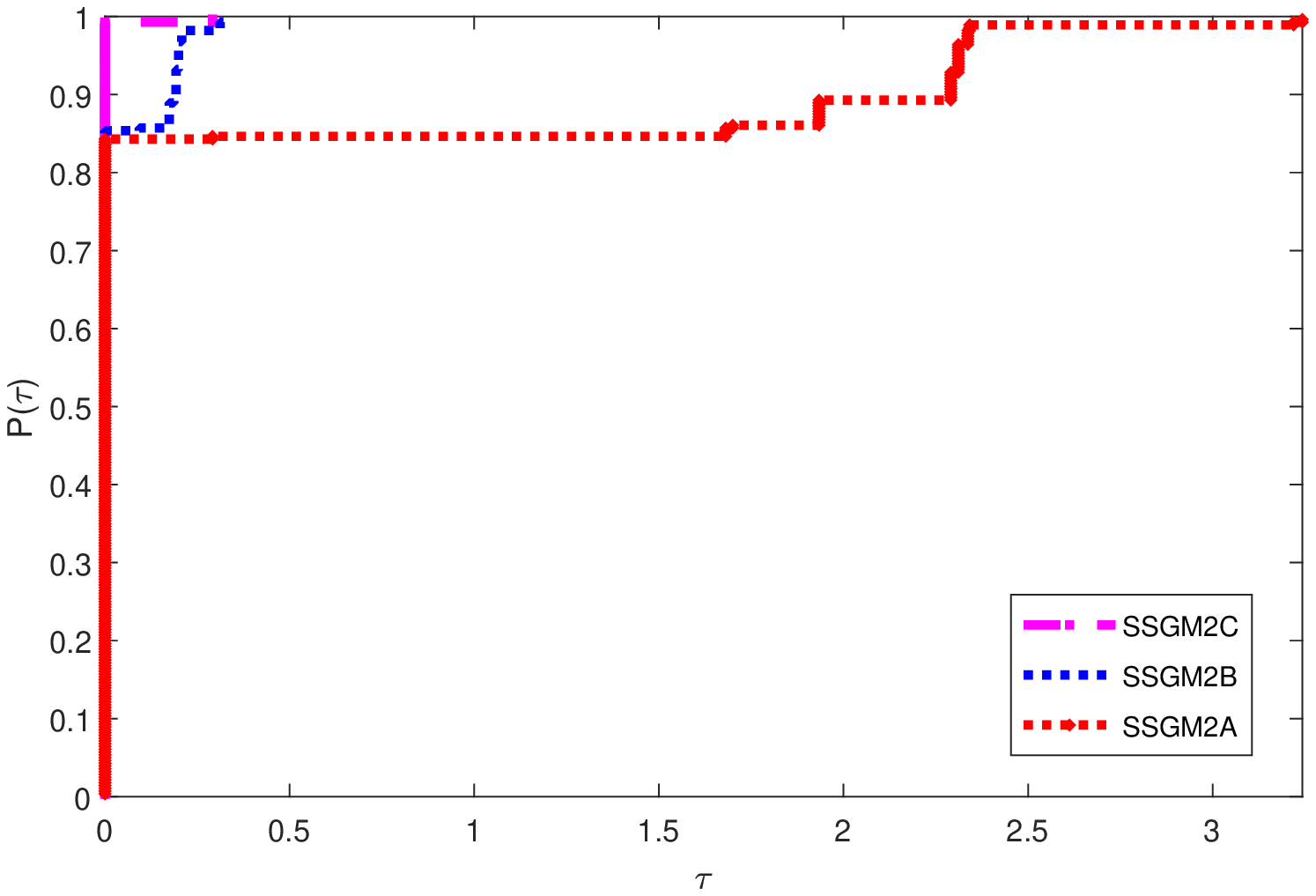}
\caption{Performance profile of SSGM2A, SSGM2B and SSGM2C
methods with respect to number of function evaluation}
\label{fig5}       
\end{figure*}

\begin{figure*}
\includegraphics[width=10cm]{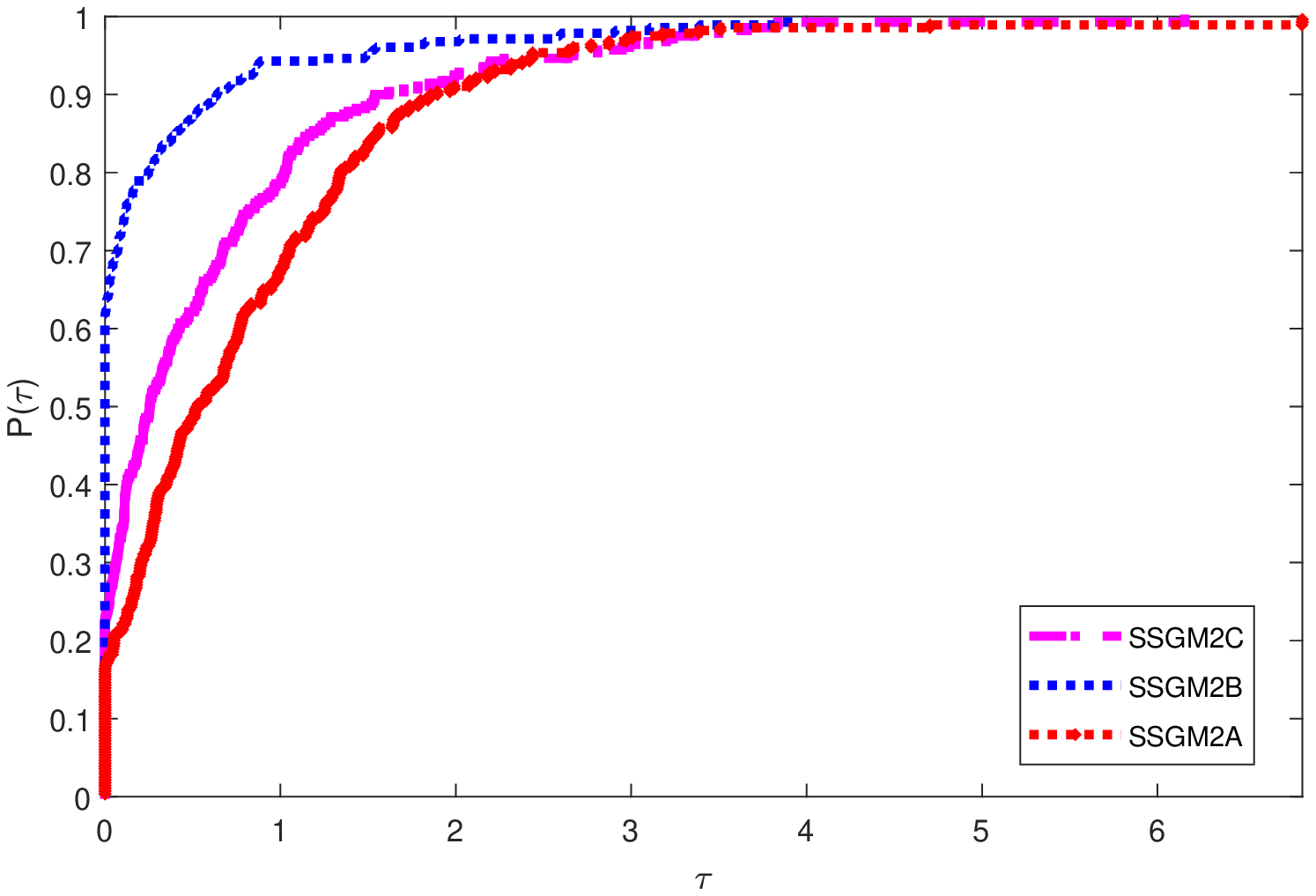}
\caption{Performance profile of SSGM2A, SSGM2B and SSGM2C
methods with respect to CPU time}
\label{fig6}       
\end{figure*}

\begin{figure*}
\includegraphics[width=10cm]{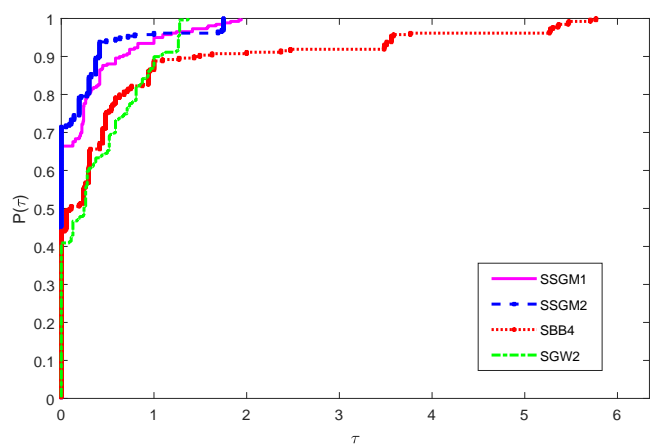}
\caption{Performance profile of SSGM1, SSGM2, SBB4 and SGW2
methods with respect to number of iterations}
\label{fig7} 
\end{figure*}

\begin{figure*}
\includegraphics[width=10cm]{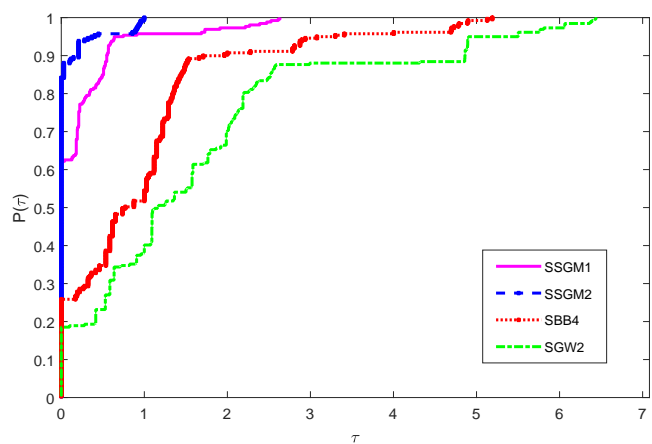}
\caption{Performance profile of SSGM1, SSGM2, SBB4 and SGW2
methods with respect to number of function evaluation}
\label{fig8} 
\end{figure*}

\begin{figure*}
\includegraphics[width=10cm]{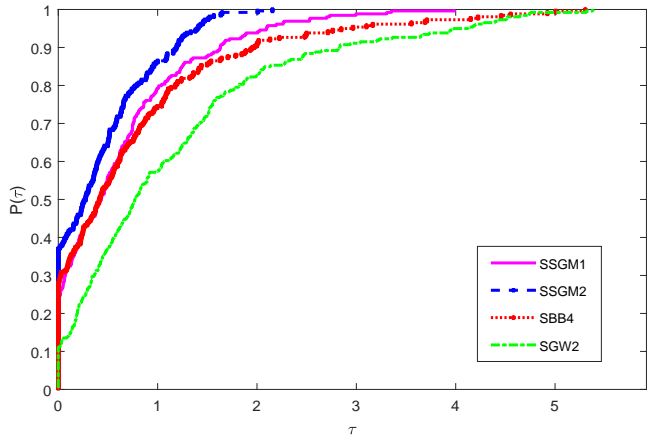}
\caption{Performance profile of SSGM1, SSGM2, SBB4 and SGW2
methods with respect to CPU time}
\label{fig9} 
\end{figure*}

\begin{table}[htbp]
  \centering
  \caption{ List of small-scale test problems with reference, dimension, starting point and residual size}
  \resizebox{\columnwidth}{!}{%
    \begin{tabular}{clccc}
    \toprule
    \textbf{S/N} & \textbf{PROBLEM NAME} & \textbf{DIMENSION (n,m)} & \textbf{INITIAL POINT} & \textbf{RESIDUAL SIZE} \\
    \midrule
    \textbf{1} & Bard Function \cite{more1981} & (3,15)  & $(-1e^k,-1e^k,-1e^k)^T, k\ge 3$ & Large \\
    \textbf{2} & Brown and Dennis \cite{more1981} & (4,20)  & $(25,5,-5,-1)^T$ & Large \\
    \textbf{3} & Beale Function \cite{more1981} & (2,3)   & $(1,1)^T$ & Zero \\
    \textbf{4} & Branin Function \cite{betts76} & (2,2)   & $(0,5)^T$ & Small \\
    \textbf{5} & Brown Badly Scaled \cite{more1981} & (2,3)   & $(1,1)^T$ & Zero \\
    \textbf{6} & Freudenstein \& Roth Function  \cite{more1981} & (2,2)   & $(0.5,-2)^T$ & Large \\
    \textbf{7} & Jennrich and Sampson \cite{more1981} & (2,10)  & $(0.2,0.3)^T$ & Large \\
    \textbf{8} & Linear Rank 1 with zero c \cite{more1981} & (10,10) & $(1,1,...,1)^T$ & Small \\
    \textbf{9} & Linear Rank 2 \cite{bb8} & (10,10) & $(1, 1/n,1/n,...,1/n)^T$ & Zero \\
    \textbf{10} & Rank Deficient Jacobian \cite{douglass16}  & (2,3)   & $(-1.2,1)^T$ & Small \\
    \bottomrule
    \end{tabular}
    }
  \label{tab1}%
\end{table}%

\begin{table}[htbp]
  \centering
  \caption{List of large-scale test problems with reference, starting point and residual size}
  \resizebox{\columnwidth}{!}{%
    \begin{tabular}{clcc}
    \toprule
    \textbf{S/N} & \textbf{PROBLEM NAME} & \textbf{INITIAL POINT} & \textbf{RESIDUAL SIZE} \\
    \midrule
    \textbf{11} & Ascher and Russel Boundary Value Problem \cite{lukvsan03} & $(1/n,1/n,...,1/n)^T $ & Zero \\
    \textbf{12} & Brown Almost Linear \cite{more1981}  & $(0.5,0.5,...,0.5)^T$ & Zero \\
    \textbf{13} & Broyden Tridiagonal Function \cite{lukvsan03} & $(-1,-1,...,-1)^T$ & Zero \\
    \textbf{14} & Discerete Boundary Value \cite{more1981} & $(1/(n+1)(1/(n+1)-1,...,1/(n+1)(1/(n+1)-1)^T$ & Zero \\
    \textbf{15} & Exponential Function 1 \cite{bb8} & $(n/n-1,n/n-1,...,n/n-1)^T$ & Zero \\
    \textbf{16} & Exponential Function 2 \cite{bb8} & $(1/n^2,1/n^2,...,1/n^2)^T$ & Zero \\
    \textbf{17} & Extended Cube Function \cite{jamil13} $(n=2k)$ & $(-2,1,-2,1,...-2,1)^T$ & Zero \\
    \textbf{18} & Extended Freudenstein \& Roth \cite{bb8} $(n=2k)$ & $(6,3,6,3,...,6,3)^T$ & Zero \\
    \textbf{19} & Extended Hemmelblau \cite{jamil13} $(n=2k)$ & $(1,1/n,1,1/n,...,1,1/n)^T$ & Zero \\
    \textbf{20} & Extended Powell Singular \cite{bb8} & $1.5\times 10^{-4}(1,1,...,1)^T$ & Zero \\
    \textbf{21} & Extended Rosenbrock\cite{more1981} $(n=2k)$ & $(-1.2,1,-1.2,1,...,-1.2,1)^T$ & Zero \\
    \textbf{22} & Extended Wood Problem \cite{ziliri81}$(n=2k)$ & $(0,0,...,0)^T$ & zero \\
    \textbf{23} & Function 21  \cite{bb8} $(n=3k)$ & $(1,1,...,1)^T$ & Zero \\
    \textbf{24} & Function 27 \cite{bb8} & $(100,1/n^2,1/n^2,...,1/n^2)^T$ & Zero \\
    \textbf{25} & Generalized Broyden Tridiag \cite{lukvsan03} & $(-1,-1,...,-1)^T$ & Zero \\
    \textbf{26} & Linear Function Full Rank \cite{more1981} & $(1,1,...,1)^T$ & Small \\
    \textbf{27} & Linear Rank 1 Function \cite{more1981} & $(1,1,...,1)^T$ & Large \\
    \textbf{28} & Logarithmic Function  \cite{bb8} & $(1,1,...,1)^T$ & Zero \\
    \textbf{29} & Penalty Function I \cite{bb8} & $(2,2,...,2)^T$ & Zero \\
    \textbf{30} & Problem 202 \cite{lukvsan03} & $(2,2,...,2)^T$ & Zero \\
    \textbf{31} & Problem 206 \cite{lukvsan03} & $(1/n,1/n,...,1/n)^T$ & Zero \\
    \textbf{32} & Problem 212 \cite{lukvsan03} & $(0.5,0.5,...,0.5)^T$ & Zero \\
    \textbf{33} & Singular Broyden \cite{lukvsan03} & $(-1,-1,...,-1)^T$ & Zero \\
    \textbf{34} & Singular Function \cite{bb8} & $(1,1,...,1)^T$ & Zero \\
    \textbf{35} & Strictly Convex Function I  \cite{bb8}& $(1/n,2/n,...,1)^T$ & Large \\
    \textbf{36} & Strictly Convex Function II \cite{bb8} & $(1/10n,2/10n,...,1/10n)^T$ & Large \\
    \textbf{37} & Trigonometric Exponential System \cite{lukvsan03} & $(0.5,0.5,...,0.5)^T$ & Zero \\
    \textbf{38} & Trigonometric Function \cite{more1981} & $(1/10n,1/10n,...,1/10n)^T$ & Zero \\
    \textbf{39} & Trigonometric Logarithmic Function & $(1,1,...,1)^T$ & Zero \\
    \textbf{40} & Variably Dimensioned \cite{more1981}  & $(1-1/n,1-2/n,...,0)^T$ & Zero \\
    \bottomrule
    \end{tabular}%
    }
  \label{tab2}%
\end{table}%

After eliminating all the problems in which at least one method failed, $258$ problems remained. Therefore we use this number of success to plot the performance profile graphs. Figures \ref{fig7} to \ref{fig9} consist of the performance profiles of SSGM1, SSGM2, SBB4 and SGW2 in terms of the number of iterations, the number of function evaluations and CPU time metric. 
 
 It can be observed from Figure \ref{fig7} , SSGM2 solves and wins $70\%$, SSGM2 solves and wins about $67\%$, while SBB4 and SGW2 solves and wins less than $50\%$ of the $258$ problems. Although, SGW2 performance is the same as SSGM1 and SSGM2 within the fraction of $\tau >1$, in most of the values of $\tau$ SBB4 performance is far less of the rest of the method except for the fraction of $\tau >5$.

In terms of the number of function evaluation (Figure \ref{fig8}), the top performer is SSGM2, as it solves about $90\%$ of the problems with the least number of functions evaluation. This good performance of SSGM2 is related to the fact that the stepsize choice incorporates more information about the Hessian of the objective function. Furthermore, this good performance indicates that the computed direction using SSGM2 stepsize is mostly acceptable. SSGM1 is the second performer in this metric with over $60\%$ success, followed by SBB4 and SGW2 both having less than $30\%$ success. 

Figure \ref{fig9} reveals that SSGM2 is very efficient because it continues to be at the top the graph. SGW2 is the slowest method with the least success of less than  $20\%$.

The details of problem $39$ of Table \ref{tab2} is given below. It is a modification of the Logarithmic function \cite{bb8}., the $28th$ problem of Table \ref{tab2}. 
 \begin{itemize}
\item [28.] Trigonometric Logarithmic Function:
\begin{equation*}
F_i(x)=\text{ln}(x_i+1)-\frac{\sin(x_i)}{n}, ~\text{ for  } i=1,2,...,n. 
\end{equation*}
\end{itemize}

\section{Conclusions and Future Research}
We have proposed a structured spectral gradient method for solving nonlinear least squares problems (SSGM). The proposed method requires neither the exact Jacobian computation nor its storage space; instead, it requires only a loop-free subroutine of the action of its transpose upon a vector. This is an advantage, especially for large-scale problems. In addition, our proposed method is suitable for both large and small residual problems. Moreover, we proposed a simple strategy with theoretical support (see Lemma \ref{lm2}) and numerical backing for safeguarding the negative curvature direction.  Despite the fact that the simple technique was shown to be effective based on the numerical experiments presented, we feel that there is room for developing better choices for the stepsize in the case of negative curvature direction. 

To the best of our knowledge, this is the first time an attempt is made to approximate the Hessian of the objective function of the nonlinear least squares problems using a spectral parameter that incorporates the special structure of the objective function. Although the approach is not new for general unconstrained optimization problems, we have addressed a special case of the unconstrained optimization problems, that is minimizing sums of squares of nonlinear functions. Considering the fact that several modifications of the two-point stepsize methods for general unconstrained optimization have been studied in the literature, we believe that this study uncovered an interesting technique for exploiting the special structure of the nonlinear least squares problems.

In the SSGM algorithm, the direction is always negative of the gradient. Consequently, a suitable line search strategy can be used for the globalization of the algorithm. Based on the preliminary numerical results presented in Section 4, we use it as evidence to claim that SSGM with the second choice of stepsize (SSGM2) is the most efficient compared to the other choices of the step size when applied to solve the nonlinear least squares problems.

In future research, we intend to study better techniques to use for getting the best approximation of the Hessian matrix that takes into account its special structure and to develop the local convergence analysis of the SSGM algorithm. Because of the low memory requirement of the SSGM method, we hope it will perform well when applied to solve practical data fitting problems and imaging problems, this is also another subject for future research. Furthermore, SSGM method can be easily modified to address general large-scale nonlinear systems of equations directly without the least square framework.

\begin{acknowledgements}
\begin{itemize}
    \item We thank Professor Sandra Augusta Santos of the State University of Campinas (UNICAMP), Brazil, for providing us with useful comments and suggestions that we used for improving the earlier version of this manuscript.
    \item The first author is currently a sandwich PhD student at UNICAMP, SP, Brazil supported by the Tertiary Education Trust Fund (TETFund) Academic Staff Training and Development (AST\& D) intervention.
\end{itemize}
\end{acknowledgements}

 \bibliographystyle{jamc}

\bibliography{ssgm}  

\end{document}